\documentclass[11pt, reqno]{amsart}
\usepackage{amsmath, amsthm, amscd, amsfonts, amssymb, graphicx, color, stmaryrd, mathrsfs} 
\usepackage[all,cmtip]{xy}	
\usepackage[bookmarksnumbered, colorlinks, plainpages]{hyperref}
\usepackage{float}	
\usepackage{tikz-cd}

\tikzset{%
    symbol/.style={%
        ,draw=none
        ,every to/.append style={%
            edge node={node [sloped, allow upside down, auto=false]{$#1$}}}
    }
}

\def\presuper#1#2{\mathop{}\mathopen{\vphantom{#2}}^{#1}\kern-\scriptspace#2}

\newcommand{\Nn}{\mathbb{N}}
\newcommand{\Rr}{\mathbb{R}}
\newcommand{\Cc}{\mathbb{C}}
\newcommand{\Zz}{\mathbb{Z}}

\newcommand{\Dom}{\mathrm{dom}}	
\newcommand{\F}{\mathcal{F}}
\newcommand{\K}{\mathcal{K}}	
\renewcommand{\L}{\mathcal{L}}	
\newcommand{\E}{\mathcal{E}}	
\newcommand{\G}{\mathcal{G}}	
\newcommand{\A}{\mathcal{A}} 	
\newcommand{\U}{\mathfrak{A}}  	
	
\newcommand{\End}{\mathrm{End}}	

\newcommand{\Hom}{\mathrm{Hom}}

\newcommand{\Op}{\mathrm{Op}}

\newcommand{\Ff}{\F_{\mathrm{fib}}}	

\renewcommand{\L}{\mathcal{L}} 

	% ...

		% ...
\newcommand{\Diff}{\mathrm{Diff}}	
\newcommand{\supp}{\mathrm{supp}}	

\newcommand{\Id}{\operatorname{Id}}	

\newcommand{\Aut}{\mathrm{Aut}}

\newcommand{\scal}[2]{\langle #1, #2 \rangle}	

\newcommand{\dtimes}{\bar{\times}}
\newcommand{\pa}{\mathrm{pa}}
\newcommand{\Exp}{\mathrm{Exp}}

\newcommand{\iso}{\xrightarrow{\sim}}	

\newtheorem{Thm}{Theorem}[section]
\newtheorem{Lem}[Thm]{Lemma}
\newtheorem{Prop}[Thm]{Proposition}

\theoremstyle{definition}
\newtheorem{Def}[Thm]{Definition}

\newtheorem{Rem}[Thm]{Remark}

\begin{document}
\setcounter{page}{1}


\title[Volterra Calculus]{A Volterra Calculus for Lie Groupoids}


\author[Karsten Bohlen]{Karsten Bohlen}







\begin{abstract}
A pseudodifferential Volterra calculus for inverting parabolic differential equations on Lie groupoids is introduced. This enables the study of fundamental solutions of various cases of heat flows on singular manifolds with corners with non-resonant boundary indicial symbols, such as the $b$-manifolds, as well as other geometric bisection covariant heat flows. We also establish the short time asymptotic expansion for the heat kernel of a positive, elliptic differential operator on a Lie groupoid that acts on suitable Sobolev Hilbert modules and is positive definite with respect to the appropriate $L^2$ inner product.
\end{abstract} \maketitle 


\section{Introduction}


The Volterra calculus was introduced by P. Greiner \cite{Greiner} and A. Piriou \cite{Piriou} in order to study the fundamental solutions of parabolic equations, such as the heat equation. This construction has been extended to various situations, cf. \cite{PongeMika}. For instance, R. Melrose extended the calculus to investigate the $b$-differential operators on compact manifolds with boundary and used it to give a heat kernel proof of the Atiyah-Patodi-Singer index theorem \cite{Melrose}. The (small) $b$-calculus of pseudodifferential operators of R. Melrose was extended to $b$-groupoids by B. Monthubert \cite{Monthubert} and there followed many similar groupoid constructions for other types of singular manifolds. We refer to \cite{DebordSkandalis} for a unified approach to blowup constructions on Lie groupoids. A pseudodifferential calculus for Lie groupoids has been described by V. Nistor, A. Weinstein and P. Xu \cite{NWX}. Our aim here is to introduce a Volterra calculus on a given Lie groupoid that enhances the existing pseudodifferential calculus to deal with parabolic equations, stated with respect to elliptic right invariant differential operators. In particular this enables a study of the heat semigroup defined on a Lie groupoid and to establish its short time asymptotic expansion. It should be noted that there are possible logarithmic terms that may arise in the asymptotic expansion, which can be local in origin or arise from the boundary resonances of the indicial symbols on a manifold with corners, especially in more singular cases. Secondly, there is a constraint involving possible limited transverse regularity. The issue of the regularity of the heat semigroup on a Lie groupoid has been studied e.g. by B. K. So \cite{So} and in the case of regular folations by J. L. Heitsch \cite{Heitsch}. With the exception of a few cases such as that of $b$-type \cite{Melrose} and edge type singularities \cite{Albin} on manifolds with boundary, the holonomy groupoid of regular foliations \cite{Heitsch}, as well as asymptotically Euclidean boundary groupoids \cite{So}, it is generally not known if the heat kernel can always be smoothed out in the transverse direction. The third issue that one encounters in the construction of pseudodifferential calculi on Lie groupoids is that the inverse operator of invertible, elliptic elements of the calculus may not be contained in the calculus due to the required proper support conditions. In the case of the Volterra calculus a suitable choice of sufficiently well-behaved smoothing ideal ensures the correct smoothing behavior outside of the diagonal. Recently, in work due to E. van Erp and R. Yuncken \cite{vEY} a groupoid approach to pseudodifferential operators is described; compare also the work of Debord-Skandalis, cf. \cite{DebordSkandalis}. The idea is to access an element in the pseudodifferential calculus on a manifold by lifting it to a family of operators on the tangent groupoid that glues, in a precise sense, the operator to its principal symbol. This vantage point allows also an elegant treatment of extensions to the standard pseudodifferential calculus that involve a filtration of the tangent space or Lie algebroid, such as the anisotropic operators that constitute the Volterra heat calculus. Such a calculus may have applications to the study of certain bisection covariant longitudinal heat flows, cf. \cite{Grad}. We remark however that recovering the Volterra calculus described herein as a particular instance of E. van Erp and R. Yuncken is not possible, due in large part to the aforementioned difficulties. However, one is able to obtain a heat calculus this way in the case of Lie groupoids over a compact base that arise from Lie group actions, cf. \cite{SLLS}.

\subsection*{Overview}

We fix a compact manifold with corners $M$ of dimension $d \in \Nn$ and a Lie groupoid $\G \rightrightarrows \G_0 = M$ with a vector bundle $E \to M$. 
For the compact manifold with corners $M$, we consider an induced filtration of the Lie algebra of vector fields on $M$ that is of a parabolic signature. 

We denote by $\overline{r}, \overline{s} \colon \G \dtimes \Rr \rightrightarrows M \times \{0\} = M$ the groupoid formed as product with the additive group $(\Rr, +)$. The residual term used below is the causal sector $S_{\rho, +}(\G \dtimes \Rr)$ of the weight-smooth algebra introduced in Definition \ref{Def:WeightSmoothVolterra}. We set
\begin{align}
\widehat{\Psi}_{ai}^{m}(\G \dtimes \Rr, \overline{r}^{\ast} E) &:= \Psi_{ai}^{m}(\G \dtimes \Rr, \overline{r}^{\ast} E) + S_{\rho, +}(\overline{r}^{\ast} E), \label{LargeVolterraCalculus}
\end{align} 
The ambient smoothing algebra is $\Psi_{\G \dtimes \Rr}^{-\infty}$ of \cite{Vassout}. Defined through the longitudinal Sobolev scale, it lies in $C_r^{\ast}(\G \dtimes \Rr)$. The definition of the weight-smooth algebra $S_{\rho}(\G)$ is an instance of the construction of differential operators on $C^{\ast}$-algebras introduced in \cite{Mohsen}. The $C^{\ast}$-algebraic functional-calculus theorem proven there furnishes the spectral invariance of the weight-smooth algebra. The reduced kernels can be identified as orbitwise smooth families with continuous transverse dependence, so
\[
\Psi_{\G \dtimes \Rr}^{-\infty} \subset C_{\mathrm{orb}}^{\infty, 0}(\G \dtimes \Rr) \cap C_r^{\ast}(\G \dtimes \Rr), 
\] 
by \cite[Thm. 10]{LV}. Since $S_{\rho, +}(\G \dtimes \Rr)$ is contained in this ideal, the extended calculus satisfies
\begin{align}
\widehat{\Psi}_{ai}^m(\G \dtimes \Rr, \overline{r}^{\ast} E) &\subset C^{\infty, 0}(\G \dtimes \Rr, \overline{r}^{\ast} E) \cap C_r^{\ast}(\G \dtimes \Rr, \overline{r}^{\ast} E). \label{LongContInclusion} 
\end{align}

A further refinement of this inclusion is possible when we utilize the Volterra property and the Paley-Wiener theorem. Then we get an inclusion into the so-called Paley-Wiener algebra intersected with the class of $C^{\infty, 0}$ kernels, by virtue of an application of the fiberwise Fourier transform. We return to this point at a later stage. Moreover, we will exhibit for each pseudodifferential operator $T \in \widehat{\Psi}_{ai}^{m}(\G \dtimes \Rr, \overline{r}^{\ast} E)$ a $1$-parameter family $\{T_{\hbar}\}_{\hbar > 0}$ with $T_{1} = T$. To that end we construct a suitable deformation groupoid for our filtration of parabolic signature which constitutes the proper weighted generalization of the tangent groupoid, cf. \cite{vEY}. This groupoid with respect to the filtration $H^{\bullet}$ is given as a set as 
\begin{align}
(\G \dtimes \Rr)^{\pa} &= [(\G \dtimes \Rr) \times \Rr_{\hbar}^{\ast}] \cup [(\G \dtimes \Rr)_{H} \times \{0\}_{\hbar}] \rightrightarrows M \times \Rr_{\hbar}. \label{DeformationGroupoid}
\end{align}

We utilize throughout this work the generalized exponential map to exhibit a localizer and a so-called standard system of coordinates wherein which we can express elements of the small compactly supported calculus, while elements of the Vassout ideal are understood via the right-regular representation. The important properties of the calculus, such as the closedness with respect to compositions and the parametrix construction, can then be verified in the standard coordinates. 




\section{The heat semigroup}


We recall the construction of the heat semigroup on a Lie groupoid, cf. \cite{BLS}. Let $\E := C_r^{\ast}(\G, r^{\ast} E)$ be the $C_r^{\ast}(\G)$ module. We also recall the definition of the Sobolev Hilbert modules, cf. \cite{Skandalis_course_2015}. 
Denote the calculus of pseudodifferential operators with compactly supported Schwartz kernels by $\Psi_{\G, c}^{\ast}$. Fix the \emph{generator} $P \in \Psi_{c, \G}^{1}$ and define the scale of Sobolev modules as the $C_r^{\ast}(\G)$-modules $H^t := \overline{C_c^{\infty}(\G)}^{\scal{\cdot}{\cdot}_t}, \ t \in \Rr$, where 
\[
\scal{f}{g}_t = \scal{(\Id + P^2)^t f}{g} \in C_r^{\ast}(\G), 
\]

and $\scal{a}{b} = a^{\ast} b$ is the inner product on the right. Given $Q \in \Psi_{\G, c}^m$ we have that $Q \in \L(H^t, H^{t-m})$, where the latter class of so-called \emph{morphisms} are the $C_r^{\ast}(\G)$-linear maps that admit an adjoint, or equivalently, the $C_r^{\ast}(\G)$-linear maps that have orthocomplemented graphs. Set:
\begin{align}
\|Q\|_{t\to t-m} := \|(1+P^2)^{(t-m)/2}Q(1+P^2)^{-t/2}\|_{\L(\E)}. \label{SovolevNorm}
\end{align}

Define the Vassout smoothing operators $\Psi_{\G}^{-\infty}$ (cf. \cite{Vassout}) via
\begin{align}
\E^{\infty} &:= \bigcap_k \Dom(P^k), \ \Psi_{\G}^{-\infty} := \E^{\infty} \cap (\E^{\infty})^{\ast}. \label{VassoutIdeal} 
\end{align}
Note that $\Psi_{\G, c}^{0}$ constitutes a subalgebra of the multiplier algebra $M(C_r^{\ast}(\G))$ and $\Psi_{\G, c}^{<0}$ is contained in $C_r^{\ast}(\G)$. In addition,
\[
H^t = \overline{\Psi_{\G, c}^{<-t}}^{\|\cdot\|_t}, \ \E^{\infty} = \bigcap_{t} H^t \subset C_r^{\ast}(\G). 
\]

Let $\F_1(M)$ denote the set of boundary hypersurfaces and choose defining functions $\rho_F$. We assume that the chosen integration is compatible with these functions in the following sense. For every $F \in \F_1(M)$, the ratio $r^{\ast} \rho_F / s^{\ast} \rho_F$ extends to a positive smooth function on $\G$ in cases of interest, cf. Remark \ref{Rem:WeightSmoothVolterra} below. On $\G \dtimes \Rr$ put
\[
l_F(\gamma) := \log \frac{\rho_F(r(\gamma))}{\rho_F(s(\gamma))}.
\]

The function $l_F$ is a smooth groupoid $1$-cocycle. Multiplication of a compactly supported regularizing kernel by $l_F$ defines a derivation $\delta_F = [\log \rho_F, -]$. We use the same notation for its closure in the topology of the smoothing algebra on the product groupoid. 

\begin{Def}
The \emph{weight-smooth algebra} is 
\[
S_{\rho}(\G \dtimes \Rr, \overline{r}^{\ast} E) := \left\{R \in \Psi_{\G \dtimes \Rr}^{-\infty}(\overline{r}^{\ast} E) : \delta_{\rho}^{\alpha} R \in \Psi_{\G \dtimes \Rr}^{-\infty}(\overline{r}^{\ast} E) \ \text{for every} \ \alpha \in \Nn_0^{\F_1(M)}\right\},
\]

where
\[
\delta_p^{\alpha} := \prod_{F \in \F_1(M)} \delta_F^{\alpha_F}.
\]

Its Fr\'echet topology is generated by the graph seminorms
\[
p_{N,k}(R) := \sum_{|\alpha| \leq k} \max_{\substack{s,t \in \Zz \\ |s|, |t| \leq N}} \|\delta_{\rho}^{\alpha} R\|_{\L(H^s, H^t)}, \ N, k \in \Nn_0,
\]

formed with the Sobolev modules of $\G \dtimes \Rr$. The causal sector is
\[
S_{\rho, +}(\G \dtimes \Rr, \overline{r}^{\ast} E) := \{R \in S_{\rho}(\G \dtimes \Rr, \overline{r}^{\ast} E) : \supp \ k_R \subset \{t \geq 0\}\}.
\]

We define $S_{\rho, -}$ in the same way with support in $\{t \leq 0\}$. If $\F_1(M)$ is empty, the definition gives the smoothing algebra on the product groupoid. 
\label{Def:WeightSmoothVolterra}
\end{Def}

\begin{Rem}
Let $(\A, \varrho)$ be an almost injective Lie algebroid over a manifold with embedded corners $M$, i.e. over the interior $M_0 = M \setminus \partial M$ we have $\A_{|M_0} \cong TM_0$ via the anchor. Let $\F = \varrho(\Gamma(\A))$ be the corresponding singular foliation. Then $\F$ is projective, and its holonomy groupoid is a Lie groupoid integrating $\A$. In particular, this is the minimal source-connected integration, i.e. every source-connected integration maps into it, cf. \cite{D}. Fix a boundary hypersurface $H$ and a defining function $\rho_H$. Assuming that the vector fields $\Gamma(\A)$ tangent to the hyperfaces, we obtain the algebroid cocycle
\[
c_H(X) = \frac{\varrho(X)(\rho_H)}{\rho_H}, \qquad c_H \in Z^1(\A). 
\]

There is a canonical linear holonomy representation $\Psi \colon \mathcal{H}(\F)_{|L} \to \mathrm{Iso}(NL)$ on the normal spaces to each leaf, \cite{AZ}. The corresponding infinitesimal representation is the Bott representation. Since every vector field in $\F$ is tangent to $H$, its local flow preserves $H$. Any holonomy transformation $\phi \colon (S_x, x) \to (S_y, y)$ associated with an arrow $\gamma \colon x \to y$ such that $x, y \in H$ satisfies $\phi(S \cap H) \subset H$. Consequently, its differential induces an isomorphism on the normal line $d \phi \colon N_x H \to N_y H$. Utilize $d \rho_H$ as the trivialization of the normal line bundle such that the induced action is multiplication by a scalar $\lambda_H(\gamma) > 0$. By the groupoid structure we obtain $\lambda_H(\gamma_1 \gamma_2) = \lambda_H(\gamma_1) \lambda_H(\gamma_2)$ for composable arrows. Thereby, $l_H := \log \lambda_H$ is an additive groupoid cocycle. We obtain positivity when the holonomy maps arise from flows tangent to $H$ that preserve the inward side of the boundary (a common useful assumption). Their derivative in the normal direction then belongs to $\Rr_{+}$. If $X \in \Gamma(\A)$ generates a local flow $\Phi_t$, write 
\[
\rho_H(\Phi_t(x)) = \lambda_t(x) \rho_H(x). 
\]

Differentiating at $t = 0$, 
\[
\dot{\lambda}_0(x) = \frac{\varrho(X)(\rho_H)}{\rho_H}(x) = c_H(X)(x).
\]

Since $\lambda_0 = 1$, 
\[
\partial_{t|0} \log \lambda_t = c_H(X). 
\]

And thus $dl_{H|\A} = c_H$. 

In addition, we can explain the descent from the source-simply-connected integration to the source-connected integration as follows. Let $\widetilde{\G} \to \mathcal{H}(\F)$ be the quotient from the source-simply-connected integration to the holonomy groupoid. The cocycle integrates uniquely to $\widetilde{l}_{H} \colon \widetilde{\G} \to \Rr$. An element of the kernel of 
\[
\widetilde{\G} \to \mathcal{H}(\F)
\]

has trivial holonomy germ. The derivative on every transverse direction is consequently the identity. In particular the action on $NH$ is multipliation by $1$. Thereby, $\exp(\widetilde{l}_H(\gamma)) = 1$. Since $\widetilde{l}_H$ is real-valued, $\widetilde{l}_H(\gamma) = 0$. Hence $\widetilde{l}_H$ vanishes on the kernel and descends $l_H \colon \mathcal{H}(\F) \to \Rr$. On the interior we have the formula 
\[
\lambda_H(\gamma) = \frac{\rho_H(r(\gamma))}{\rho_H(s(\gamma))}.
\]

By virtue of the holonomy construction we can extend the definition beyond the interior by setting $\lambda_H(\gamma)$ to equal the normal derivative of the holonomy transformation. The two descriptions of $\lambda_H$ arise from the same flow equation and therefore fit smoothly. To wit,
\[
\partial_t \rho_H(\Phi_t(x)) = c_H(X_t)(\Phi_t(x))\rho_H(\Phi_t(x)).
\]

The solution is
\[
\rho_H(\Phi_t(x)) = \exp\left(\int_0^t c_H(X_{\tau})(\Phi_{\tau}(x))\,d\tau\right) \rho_H(x).
\]

Therefore, one obtains a smooth positive multiplicative function $\mathcal{H}(\F) \to \Rr_+$ such that 
\begin{align*}
r^{\ast} \rho_H = \lambda_H s^{\ast} \rho_H, \ l_H = \log \lambda_H, \ dl_{H|\A} = \rho_H^{-1} d_{\A} \rho_H. 
\end{align*}

\begin{center}
\begin{equation*}
\begin{tikzcd}[row sep=huge, column sep=huge, text height=1.5ex, text depth=0.25ex]
\widetilde{\G} \arrow[dr, " ", "\widetilde{l}_H" '] \arrow[r, " ", " "] & \mathcal{H}(\F) \arrow[d, dashed, "l_H" '] \arrow[r, " ", "\lambda_H" '] & (\Rr_{+}, \cdot) \\
& \Rr &
\end{tikzcd}
\end{equation*}
\end{center}

\label{Rem:WeightSmoothVolterra}
\end{Rem}

\begin{Prop}
The algebra $S_{\rho}(\G \dtimes \Rr)$ is a Fr\'echet $\ast$-algebra whose unitization is spectrally invariant in the unitization of $C_r^{\ast}(\G \dtimes \Rr)$, i.e. a $\Psi^{\ast}$-algebra. The inclusions
\[
\Psi_{\G \dtimes \Rr,c}^{-\infty} \subset S_{\rho}(\G \dtimes \Rr) \subset \Psi_{\G \dtimes \Rr}^{-\infty}
\]

are continuous. If $P \in \Psi_{\G \dtimes \Rr, c}^{m}$ and $R \in S_{\rho}(\G \dtimes \Rr)$, then
\[
PR, RP \in S_{\rho}(\G \dtimes \Rr).
\]

For a Volterra operator $P$ and $R \in S_{\rho, +}$, both products belong to $S_{\rho, +}$. The causal sector is closed under convolution. Adjoint interchanges the two time sectors,
\[
S_{\rho, +}^{\ast} = S_{\rho, -}. 
\]
\label{Prop:WeightSmoothVolterra} 
\end{Prop}

\begin{proof}
The cocycle identity for $l_F$ gives the Leibniz rule
\[
\delta_F(RS) = \delta_F(R) S + R \delta_F(S).
\]

The closed graph construction makes the common domain of the commuting derivations $\delta_F$ a Fr\'echet algebra. Iteration yields
\[
\delta_{\rho}^{\alpha}(RS) = \sum_{\beta \leq \alpha} \binom{\alpha}{\beta} \delta_{\rho}^{\beta}(R) \delta_{\rho}^{\alpha - \beta}(S), 
\]

which gives continuity of multiplication. Since $l_F(\gamma^{-1}) = -l_F(\gamma)$, one has
\[
\delta_F(R^{\ast}) = -\delta_F(R)^{\ast}.
\]

The algebra is stable under adjoint. Let $a = 1 + R$ with $R \in S_{\rho}(\G \dtimes \Rr)$, and suppose that $a$ is invertible in the unitized reduced $C^{\ast}$-algebra. Holomorphic stability of the ambient smoothing algebra gives $a^{-1} = 1 + S$ with $S \in \Psi_{\G \dtimes \Rr}^{-\infty}$. Applying $\delta_F$ to $a a^{-1} = 1$ gives
\[
\delta_F(a^{-1}) = -a^{-1} \delta_F(R) a^{-1}.
\]

Repeated use of this identity expresses every $\delta_{\rho}^{\alpha}(a^{-1})$ as a finite sum of products of $a^{-1}$ with derivatives $\delta_{\rho}^{\beta}(R)$ of positive order. After writing $a^{-1} = 1 + S$, each term belongs to the smoothing algebra. Hence $a^{-1} \in S_{\rho}(\G \dtimes \Rr)^{+}$. The compactly supported kernels are dense in $C_r^{\ast}(\G)$ and belong to $S_{\rho}(\G)$, so the unitization is dense in $C_r^{\ast}(\G)^{+}$. By spectral invariance and the graph topology we obtain the $\Psi^{\ast}$ property. Compare the closed-derivation construction used in \cite[Theorem 2]{LMN}. Every compactly supported smooth kernel belongs to $S_{\rho}$, because for a compactly supported smoothing kernel, multiplication by any power of $l_F$ preserves smoothness and compact support. The continuity of the inclusion of $S_{\rho}$ in the ambient smoothing algebra follows by observing that integer Sobolev orders form a cofinal defining family for the smoothing topology; defining seminorms of $S_{\rho}$ with $k = 0$ dominate this topology. For $P \in \Psi_{\G \dtimes \Rr, c}^m$, multiplication of its conormal kernel by $l_F$ lowers the pseudodifferential order because $l_F$ vanishes on the unit section.  Thus
\[
\delta_{\rho}^{\alpha}(P) \in \Psi_{\G \dtimes \Rr, c}^{m - |\alpha|}. 
\]

The Leibniz formula now gives the bimodule statement. The time coordinate is additive under groupoid multiplication, so convolution preserves support in $\{t \geq 0\}$. The formula for the adjoint reverses the time coordinate and gives the last assertion. 
\end{proof}

Within the Volterra calculus, $S_{\rho, +}$ serves as the residual class. The ambient algebra $S_{\rho}$ carries the $\ast$-structure and spectral invariance from Sobolev Hilbert module setting, and its graph topology controls logarithmic boundary commutators in the relevant applications involving manifolds with corners. 
The compact morphisms $\E \to \F$ between Hilbert $A$-modules are defined as the linear subspace spanned by the usual rank one projections; the embedding $H^t \hookrightarrow H^{t+m}$ is always compact in that sense, for each $m > 0$. A densely defined operator $T \colon \E \to \E$ for a Hilbert $A$-module $\E$, with densely defined adjoint is regular if its graph $G(T)$ is orthocomplemented as an $A$-submodule of $\E \oplus \E$. The topology of strict continuity on $\L(\E)$ is the weakest topology generated by the seminorms $T \mapsto \|T x\|, x \in \E$ and $T \mapsto \|T^{\ast} y\|, y \in \E$. It can be checked that $\L(\E)$ specified in this way is a $C^{\ast}$-algebra. If we view a $C^{\ast}$-algebra $A$ as a right Hilbert $A$-module with $\scal{a}{b} = a^{\ast} b$, then the multiplier algebra $M(A)$ identifies with the morphisms $\L(A)$. 

Consider an element $Q \in \Psi_{\G, c}^m(r^{\ast} E)$ for $m > 0$ that is self-adjoint, elliptic as an unbounded operator $\E \to \E$ and non-negative with respect to the canonical $L^2 = H^0$ inner product. Let $\pi \colon C_0(\Rr) \to \L(\E)$ be a nondegenerate representation such that we can find an extension $\widetilde{\pi}$ from $C(\Rr)$ (i.e. the regular operators on $C_0(\Rr)$) defined via $C_0(\Rr) \otimes_{\pi} \E \cong \E$ through $\widetilde{\pi}(f) = f \otimes_{\pi} \Id$ and with the property $\widetilde{\pi}(\Id_{\Rr}) = Q$. With respect to the identifications $C_b(\Rr) \cong M(C_0(\Rr))$ and $\L(\E) \cong M(\K(\E))$, we obtain a strictly continuous homomorphism by restriction
\[
\overline{\pi} = \widetilde{\pi}_{|C_b(\Rr)} \colon C_b(\Rr) = M(C_0(\Rr)) \to \L(\E) = M(\K(\E)). 
\]

In particular, for $\lambda \in \Cc \setminus \Rr$ on has $\pi((\cdot - \lambda)^{-1}) = (Q - \lambda)^{-1}$. The extension $\overline{\pi}$ uniquely determines $\pi$ (and conversely), cf. \cite{LV}. 

\begin{Def}
Denote by $\Gamma$ the curve parametrized by $\Rr_{+} \ni t \mapsto -1 + t(1 \pm i)$. Define $f_{\lambda} \in C_0(\Rr)$ via $f_{\lambda}(x) = \alpha(x) (x - \lambda)^{-1}, \ \lambda \in \Gamma$. Here $\alpha$ is the continuous non-decreasing function defined by
\[
\alpha(x) = \begin{cases} 0, \ x < -\frac{1}{2}, \\ 1, \ x \geq -\frac{1}{4} \end{cases}.
\]

Then $\pi(f_{\lambda}) = (Q - \lambda)^{-1}$ and the \emph{heat kernel} is defined by the Dunford integral
\[
\exp(-tQ) := \frac{1}{2\pi i} \int_{\Gamma} e^{-t \lambda} (Q - \lambda)^{-1}\,d\lambda. 
\]

\label{Def:HeatSemigroup}
\end{Def}

\begin{Prop}
Let $Q \in \Psi_{\G,c}^m(r^*E)$ with $m > 0$ be symmetric on $C_c^\infty(\G ,r^*E)$, $\G$-elliptic, and nonnegative. Then its closure on $H^0$ is regular selfadjoint with $\mathrm{spec}(Q) \subset [0,\infty)$. For a given $k \in \Nn$ and any $u_0 \in H^k$, the initial value problem 
\[
(\partial_t + Q) u = 0, \ u(0) = u_0
\]

has a unique solution in $\bigcap_{0 \leq j \leq k} C^j([0, \infty), H^{k-j})$, given by $u(t) = \exp(-t Q) u_0$. In addition, the heat kernel $\exp(-t Q)$ furnishes an $H^0 = L^2$ contractive semigroup that is contained in $\Psi_{\G}^{-\infty}(r^{\ast} E)$.
\label{Prop:HeatSemigroup}
\end{Prop}

\begin{proof}
We observe that $Q$ is closable and its closure is a regular selfadjoint unbounded morphism on $\E$ to which we apply the continuous functional calculus. We have $Q \in \L(H^{s + m}, H^{s})$ for the scale of Sobolev modules $\{H^s\}$ where $H^s$ is understood as the domain of $(\Id + Q)^{\frac{s}{m}}$ with equivalent norm, by ellipticity. Denoting by $\pi : C_0(\Rr) \to \L(\E)$ the functional calculus associated with $Q$, we get $\pi(f_\lambda) = (Q-\lambda)^{-1}$ and $\|(Q-\lambda)^{-1}\|_{\E} \le \|f_\lambda\|_{\infty}\le c(1+|\Im \lambda|)^{-1}$ for some constant $c$. It follows that the integral above is absolutely convergent in $\L(\E)$, and standard calculations on this integral prove that $P^ke^{-tQ}$ and $e^{-tQ}P^k$ belong to $\L(\E)$ as well, for any $k\in\Nn$. By construction of the functional calculus, the map $\lambda \mapsto \pi(f_{\lambda})$ is strictly continuous, cf. \cite[Proposition 5.19]{Skandalis_course_2015}. By virtue of the identity
\[
\frac{(Q - \lambda)^{-1} - (Q - \mu)^{-1}}{\lambda - \mu} = \frac{(Q - \lambda)^{-1} (\mu - \lambda) (Q - \mu)^{-1}}{\lambda - \mu} = -(Q - \lambda)^{-1} (Q - \mu)^{-1} 
\]

we obtain as $\lambda \to \mu$ that the left hand side converges to $-(Q - \lambda)^{-2}$. Hence $\Cc \setminus \mathrm{spec}(Q) \ni \lambda \mapsto (Q - \lambda)^{-1}$ is holomorphic as a map $\Cc \setminus \mathrm{spec}(Q) \to \L(H^s)$. By the previous arguments we have $\exp(-t Q) \in \L(H^k)$ and $t \mapsto \exp(-t Q) f \in C(\Rr, H^{k})$ for any $f \in H^k$. Since $\frac{1}{t} (\exp(-t \lambda) - 1) \to \lambda$ as $t \to 0$ uniformly on compact subsets of $\Rr$, we get by \cite[Appendix]{Skandalis_course_2015} that $\frac{1}{t} (\exp(-t Q) - 1)$ converges to $Q$ strongly meaning
\[
\left\|\frac{1}{t} (\exp(-t Q) f - f) - Q f\right\|_{\E} \to 0, \ \text{as} \ t \to 0, \ \text{for all} \ f \in H^m.
\]

Hence $t \mapsto \exp(-t Q) f \in C^1(\Rr, \E) \cap C^0(\Rr, H^m)$ for all $f \in H^m$ and $\partial_t \exp(-t Q) f = -Q \exp(-t Q) f$ for all $f \in H^m$ and all $t \in \Rr$. Finally, we need to show that the heat kernel furnishes a semigroup and the $H^0 = L^2$ contractivity. Let $t > 0, s > 0$ and choose a second contour $\Gamma' = \Gamma + \epsilon$ obtained by shifting $\Gamma$ slightly to the right (for some fixed $\epsilon > 0$). In particular, $\Gamma$ and $\Gamma'$ are disjoint, $\Gamma, \Gamma' \subset \Cc \setminus \mathrm{spec}(Q)$, and there is a constant $c_{\epsilon} > 0$ such that $|\mu - \lambda| \geq c_{\epsilon}$ for all $\lambda \in \Gamma$ and $\mu \in \Gamma'$. Using the Dunford representations and Fubini's theorem (justified by the resolvent bound $\|(Q - \lambda)^{-1}\| \lesssim |\lambda|^{-1}$ on $\Gamma$ and exponential decay of $e^{-t \Re \lambda}$ along the rays), we compute
\begin{align*}
e^{-Q t} e^{-Q s} &= (2 \pi i)^{-2} \int_{\Gamma} \int_{\Gamma'} e^{-\lambda t} (Q - \lambda)^{-1} e^{-\mu s} (Q - \mu)^{-1} \,d\mu \,d\lambda \\
&= (2\pi i)^{-2} \int_{\Gamma} \int_{\Gamma'} e^{-\lambda t - \mu s}(\mu - \lambda)^{-1} ((Q - \lambda)^{-1} - (Q - \mu)^{-1})\,d\mu \, d\lambda \\
&= (2\pi i)^{-1} \int_{\Gamma} e^{-\lambda (t + s)} (Q - \lambda)^{-1} \,d\lambda \\
&= e^{-Q (t + s)}.
\end{align*}

In the next to final line we made use of the identities
\[
\int_{\Gamma} e^{-\lambda t} (\mu - \lambda)^{-1}\,d\lambda = 0, \ \int_{\Gamma'} e^{-\mu s} (\mu - \lambda)^{-1}\,d\mu = 2\pi i e^{-\lambda s}.
\]

By way of justification, observe that for fixed $\lambda \in \Gamma$, the function $\mu \mapsto e^{-\mu s} (\mu - \lambda)^{-1}$ is holomorphic in $\mu$ and to the right of $\Gamma'$ (the only pole is at $\mu = \lambda$, which lies strictly to the left of $\Gamma'$), and the contribution from a closing arc at infinity vanishes by the factor $e^{-s \Re \mu}$. Application of Cauchy's theorem yields
\[
\int_{\Gamma'} e^{-\mu s}(\mu - \lambda)^{-1}\,d\mu = 0.
\]

For fixed $\mu \in \Gamma'$, the function $\lambda \mapsto e^{-\lambda t} (\mu - \lambda)^{-1}$ has a simple pole at $\lambda = \mu$, and this pole lies to the right of $\Gamma$; closing $\Gamma$ to the right and applying the residue theorem yields
\[
\int_{\Gamma} e^{-\lambda t} (\mu - \lambda)^{-1}\,d\lambda = -2 \pi i e^{-\mu t}.
\]

To show contractivity of the semigroup we can adapt a standard Hille-Yosida approximation argument to our setting, cf. e.g. \cite{Trotter}. We have 
\begin{align}
& (-\infty, 0) \subset \Cc \setminus \mathrm{spec}(Q), \label{HYRes} \\
& \|\lambda (Q + \lambda)^{-1}\| \leq 1  \label{HYBound}
\end{align}
for the regular selfadjoint operator $Q \geq 0$. Set 
\[
Q_{\lambda} := \lambda \Id - \lambda^2 (Q + \lambda)^{-1} \in \L(\E), \ \lambda > 0
\] 
and note that $Q_{\lambda} = \lambda Q (Q + \lambda)^{-1}$. Since the approximant is a bounded operator in $\L(\E)$, we obtain by virtue of the functional calculus of the $C^{\ast}$-algebra $\L(\E)$ the estimate 
\[
\|e^{-t Q_{\lambda}}\| \leq e^{-\lambda t} e^{\| \lambda^2 (Q + \lambda)^{-1}\| t} \leq e^{-\lambda t} e^{\lambda t} = 1. 
\]

For each $t \geq 0$ define the semigroup 
\[
S_t x := \lim_{\lambda \to \infty} e^{-t Q_{\lambda}} x
\] 
for $x \in \E$. One checks the uniform convergence on bounded intervals using \eqref{HYBound} which furnishes a contraction, since each $e^{-t Q_{\lambda}}$ is a contraction. One then exhibits $Q$ as a generator of $(S_t)_{t \geq 0}$ using \eqref{HYRes}, which by virtue of uniqueness of the solution to the heat equation shows that $S_t$ reproduces the heat kernel. 
\end{proof}

\section{Anisotropic Operators}




The aim is to introduce a Volterra calculus for parabolic PDE posed on a Lie groupoid. We provide a parametrix construction, so that the fundamental solutions of weakly parabolic PDE, stated with respect to the boundary tangential vector fields on the corresponding Lie algebroid, are contained within the calculus. 





The anisotropic operators arise from quantization of continuous families of Volterra symbols. These symbol classes are families of symbol spaces along the fibers of a submersion. Consider the source map $\overline{s} \colon \G \dtimes \Rr \to (\G \dtimes \Rr)_0 = M$ defined by $\overline{s}(\gamma, t) = (s(\gamma), 0)$ and typical fiber $\overline{s}^{-1}(x) = \G_x \times \Rr$. The Lie algebroid is $\A(\G \dtimes \Rr) \cong \A(\G) \oplus \Rr_M$ where $\Rr_M = M \times \Rr$ is the trivial bundle over $M$. 
A groupoid parametrization of $\overline{r}, \overline{s} \colon \G \dtimes \Rr \rightrightarrows M$, centered at $x_0 \in M$, is given by an open set $\G \dtimes \Rr \supset \Omega \cong U_{\overline{r}, \overline{s}} \times V$ for $U_{\overline{r}, \overline{s}} =: U \subset M$ open and the data $(\varphi, \psi)$ where $\varphi \colon U \to M, \ \psi \colon U \times V \to \G$ with the properties, cf. \cite{LR}
\begin{align*}
&\psi(0, 0) = x_0, \ r(\psi(u, v)) = \varphi(u) \\
&\psi(U \times \{0\}) = \psi(U \times V) \cap M,
\end{align*}

where the first two conditions imply $\varphi(u) = \psi(u, 0)$.

\begin{Def}
Let $\Cc_- := \{\tau \in \Cc : \Im \tau < 0\}$ and write the parabolic dilations as $\delta_{\lambda}(\xi, \tau) := (\lambda \xi, \lambda^2 \tau)$ for $\lambda > 0$ on the fibers of $\A^{\ast} \oplus \Rr_M$. Set $\|(\xi, \tau)\|_{ai} := (|\xi|^2 + |\tau|)^{\frac{1}{2}}$.  

Fix a groupoid parametrization $\Omega \cong U \times V$ of $\overline{r}, \overline{s} \colon \G \dtimes \Rr \rightrightarrows M$, with $U\subset M$ open and $V$ an open neighbourhood of the unit in a longitudinal chart of $\G_{x}$ (via $\psi$). For $u\in U$ write $x=\varphi(u)$ and identify the longitudinal slice
\[
V_x := \{y = \psi(u, v) \in \G_x : \ v \in V\}
\]

near the unit in $\G_x$.

\begin{enumerate}
\item For $m \in \Rr$, define $S_{ai, m}(\A^{\ast} \oplus \Rr_M)$ to be the space of families 
\[
q = (q_x)_{x \in M}, \qquad q_x \colon V_x \times (\A_x^{\ast} \times \Cc_-) \setminus \{0\} \to \Cc,
\]
such that, for each $x \in M$:
\begin{enumerate}
\item $q_x$ is $C^{\infty}$ in $\xi \in \A_x^{\ast} \setminus \{0\}$ and holomorphic in $\tau \in \Cc_-$. 

\item For all $\lambda > 0$, $q_x(y, \delta_{\lambda}(\xi, \tau)) = \lambda^m q_x(y, \xi, \tau)$. 

\item For every chart $\Omega \cong U \times V$ as above, the map
\[
u \mapsto q_{\varphi(u)}(\psi(u, \cdot), \cdot, \cdot)
\]

is continuous from $U$ into the Fr\'echet space of $C^{\infty}$ functions of $(y, \xi)$ on $V \times (\A_x^{\ast} \setminus \{0\})$ that are holomorphic in $\tau \in \Cc_-$. 
\end{enumerate}

\item A family $q = (q_x)_{x \in M}$ lies in $S_{ai, ph}^m(\A^{\ast} \oplus \Rr_M)$ if, in every chart $\Omega \cong U \times V$ and for each $x = \varphi(u) \in U$, there are fiberwise homogeneous components $q_{x, m-j} \in S_{ai, m-j}(\A^{\ast} \oplus \Rr_M)$ such that, for every $N \in \Nn$, compact $K \subset V$, multi-indices $\alpha$ in $y$, $\beta$ in $\xi$, and $k \in \Nn$,
\[
\left\|\partial_y^{\alpha} \partial_{\xi}^{\beta} \partial_{\tau}^{k} \left(q_{\varphi(u)}(\psi(u,v), \xi, \tau) - \sum_{j < N} q_{\varphi(u), m-j}(\psi(u,v), \xi, \tau) \right)\right\| \leq C \|(\xi, \tau)\|_{ai}^{m - |\beta| - 2k}.
\]

Each $q_{x,m-j}$ is $C^\infty$ in $(y,\xi)$, holomorphic in $\tau \in \Cc_-$, and parabolically homogeneous of degree $m-j$ in $(\xi,\tau)$. The constants $C$ may depend on $(\alpha,\beta,k,K)$ but are uniform over $v\in K$.
\end{enumerate}
\label{Def:anisotropicSymbol}
\end{Def}

Let $S(\A^{\ast} \oplus \Rr_M)$ (resp. $S(\A \oplus \Rr_M)$) denote fiberwise Schwartz functions on the vector bundle. We fix the convention for the fiberwise Fourier transform
\[
\Ff f(\xi, \tau) = (2 \pi)^{-(d+1)} \int_{\overline{\pi}(\xi, \tau) = \pi(\zeta, t)} e^{-i \scal{\xi}{\zeta} - i \tau t} f(\zeta, t)\,d\zeta. 
\]

\begin{Lem}[cf. \cite{Greiner}]
Let $m \in \Rr$ and $q \in S_{ai, m}(\A^{\ast} \oplus \Rr_M)$. Then there exists a unique family of tempered distributions
\begin{align*}
& g=(g_x)_{x \in M}, \ g_x(y; \xi, \tau) \in S'(\A_x^{\ast} \times \Rr),  
\end{align*}

such that 
\begin{enumerate}
\item[(i)] For each $x$, $g_x(y;\xi,\tau)=q_x(y,\xi,\tau)$ on
$(\xi,\tau)\neq (0,0)$.
\item[(ii)] For each $x$ and $\lambda>0$,
\[
g_x\left(y;\,\delta_\lambda(\xi,\tau)\right) = \lambda^{m}\,g_x(y;\,\xi,\tau)
\quad\text{in the sense of distributions in }(\xi,\tau).
\]
\item[(iii)] Writing
\[
K_x(y;\zeta,t)\;:=\;\Ff^{-1}[g_x](y;\zeta,t),
\]
we have $K_x(y;\zeta,t)=0$ for all $t<0$.
\item[(iv)] In any groupoid chart
$\Omega \cong U\times V$ and bundle trivialization of $\A^\ast$, all longitudinal seminorms in $(y,\xi)$ of $g_x$ (and of $K_x$ for $t>0$) depend continuously on $x\in U$.
\end{enumerate}
Moreover, the construction is local and chart independent, and it is functorial under
changes of longitudinal coordinates and bundle trivializations. If, in addition,
$q\in S^{m}_{ai,ph}(\A^\ast\oplus\Rr_M)$, then $g$ admits a
polyhomogeneous expansion $g \sim \sum_{j \geq 0} g_{m-j}$ with homogeneous components
corresponding to those of $q$, and $K=\Ff^{-1}[g]$ is Volterra (supported in $\{t \geq 0\}$)
and admits the associated asymptotic expansion as $t\downarrow 0^+$.
\label{Lem:Greiner}
\end{Lem}

\begin{proof}[Proof sketch]
For each fixed $(x,y)$ set $f(\xi,\tau) := q_x(y,\xi,\tau)$. Then $f$ is smooth on $(\A_x^{\ast}\times \Cc_-) \setminus \{0\}$, holomorphic in $\tau \in \Cc_-$, and parabolically homogeneous of degree $m$, i.e.\ $f(\lambda\xi, \lambda^2\tau) = \lambda^m f(\xi,\tau)$. Recall that $d = \mathrm{rk}\,\A_x$ and set $Q:=d+2$. Choose the minimal $j \in \Nn_0$ with $m+2j>-Q$.
Define the $j$-fold $\tau$-antiderivative
\[
f_j(\xi,\tau) := i^{j} \int_{\infty}^{\tau} \int_{\infty}^{\sigma_{j-1}} \cdots \int_{\infty}^{\sigma_1} f(\xi, \sigma_0)\, d\sigma_0 \cdots d\sigma_{j-1},
\]
where the integration runs along any piecewise $C^1$ path inside $\Cc_-$ and the constant of integration is fixed by requiring $f_j(\xi,\tau) \to 0$ as $|\tau| \to \infty$ within $\Cc_-$. Then $f_j$ is holomorphic in $\tau$, homogeneous of degree $m + 2j$, and $f_j \in L^1_{\mathrm{loc}}$ near $(\xi,\tau) = (0,0)$ because $m+2j > -Q$. Set
\[
g_x(y;\xi,\tau) := \partial_\tau^{j} f_j(\xi, \tau) \in S'(\A_x^\ast \times \Rr),
\]
which is a tempered distribution, homogeneous of degree $m$, and agrees with $f$ on $((\A_x^\ast\times\Cc_-) \setminus \{0\})$. By Paley-Wiener, $f_j$ being a boundary value from the open half-plane $\Cc_-$ implies $\F^{-1}_\tau[f_j](t)$ is supported in $\{t \geq 0\}$; applying $\partial_\tau^{j}$ corresponds to multiplication by $(it)^j$, hence preserves this support.
Equivalently, one may regularize by
\[
f_\varepsilon(\xi,\tau) := (1 + i \varepsilon \tau)^{-N} f(\xi,\tau), \ N \ \text{suff. large}, \ \varepsilon>0,
\]
so that $\supp \F^{-1}_\tau[f_\varepsilon] \subset \{t \geq 0\}$ for each $\varepsilon$, and $f_\varepsilon \to g_x(y;\cdot,\cdot)$ in $S'$ as $\varepsilon \downarrow 0$.
Let $K_x(y;\zeta,t): =\F^{-1}_{(\xi,\tau) \to (\zeta,t)}[g_x(y;\xi,\tau)]$. From the preceding paragraph, $K_x(y;\zeta,t) = 0$ for $t < 0$. This gives the claimed Volterra support after also taking the ordinary $\xi \leftrightarrow \zeta$ transform. Any two such extensions differ by a homogeneous distribution whose inverse $\tau$-Fourier transform is supported both in $\{t \geq 0\}$ and in $\{t \leq 0\}$, hence must vanish.
\end{proof}


\begin{Def}
A compactly supported anisotropic (possibly polyhomogeneous) operator $T \in \Psi_{ai, c}^{m}(\G \dtimes \Rr)$ is given by a family $T = (T_x)_{x \in M}$ such that 

\emph{i)} For all groupoid parametrizations $\Omega \cong U \times V$ with $U \subset M$ open and all $\kappa_1, \kappa_2 \in C_c^{\infty, 0}(\G \dtimes \Rr)$ with support in $\Omega$, the operator $\kappa_1 T \kappa_2$ is a compactly supported $C^{0}$-family of anisotropic (possibly polyhomogeneous) pseudodifferential operators of order $m$.

\emph{ii)} For each $\kappa_1, \kappa_2$ as above and with disjoint support, $\kappa_1 T \kappa_2 \in \Psi_{ai}^{-\infty}(\{x\} \times V)$. 
\label{Def:anisotropicOperator}
\end{Def}

Fix an algebroid connection $\nabla$ on $\A(\G)\oplus \Rr_M$ and denote the generalized exponential map by
\[
\Exp^\nabla \colon \A(\G)\oplus \Rr_M \longrightarrow \G\dtimes \Rr.
\]
Quantization is performed in exponential coordinates by the usual oscillatory integral with phase compatible with the parabolic dilations, yielding a map
\[
\Op_{ai} \colon S^{m}_{ai,ph}(\A^{\ast} \oplus \Rr_M)\ \longrightarrow\ \Psi^{m}_{ai,c}(\G \dtimes \Rr),
\]
well-defined modulo $\Psi^{-\infty}_{ai}$ and compatible with the $C^{\infty,0}$ structure.

Let $\mathcal{U} \subset \A\oplus \Rr_M$ be an open neighbourhood of the zero section
and $\mathcal{V} \subset \G \dtimes \Rr$ an open neighbourhood of the unit space such that
the generalized exponential map
\[
\Exp^{\nabla} \colon \mathcal{U} \iso \mathcal{V}
\]
is a diffeomorphism. Fix a groupoid parametrization centred at $x_0\in M$:
an open chart $\Omega \subset \G\dtimes \Rr$ with a diffeomorphism
\[
\psi \colon U\times V \iso \Omega \cap \mathcal{V},
\]
where $U \subset M$ and $V$ is an open neighbourhood of the unit in the longitudinal
coordinates of $\G_{x}$. Likewise choose an algebroid chart
\[
\theta \colon U \times W \iso \Omega \cap \mathcal{U},
\]
with $W$ an open neighbourhood of the zero section in $(\A\oplus \Rr_M)|_U$.
Define
\[
\alpha := \psi^{-1} \circ \Exp^{\nabla} \circ \theta \colon U \times W \longrightarrow U \times V.
\]
Then $\alpha(u,0) = (u,0)$, and for each $x = \varphi(u)$ the restriction
$\alpha_x \colon W_x\to V_x$ maps the fiber $(\A_x\oplus \Rr)\cap W$ diffeomorphically onto
$(\G_x\times \Rr) \cap V$ with $d\alpha_x|_{0} = \Id_{\A_x\oplus \Rr}$.
We have the commuting diagram
\[
\begin{tikzcd}[row sep=huge, column sep=huge, text height=2ex, text depth=0.5ex]
U\times W \arrow[d, "\theta"'] \arrow[r, "\alpha"] & U\times V \arrow[d, "\psi"] \\
\Omega\cap \mathcal U \arrow[r, "\Exp^\nabla"'] & \Omega\cap \mathcal V
\end{tikzcd}
\]

Write
\[
\Hom(\overline{r}^{\ast} E) := \overline{r}^{\ast} E\otimes (\overline{s}^{\ast}E)^{\ast} \longrightarrow \G \dtimes \Rr
\]
for the coefficient bundle of kernels. Shrinking $\Omega$ if necessary, we trivialize
$\Hom(\overline{r}^{\ast} E)$ over $\Omega \cap \mathcal{V}$ so that for each
$(\gamma,t) \in \Omega \cap \mathcal{V}$ there is a canonical identification
\[
\vartheta_{(\gamma,t)} \colon E_{\overline{r}(\gamma,t)} \otimes E_{\overline{s}(\gamma, t)}^{\ast}
\iso \End(E_x), \ x = \overline{s}(\gamma, t),
\]
e.g. by parallel transport along the base in $U$.

Using the charts $(\psi, \theta)$ and the trivialization $\vartheta$, we obtain a
canonical topological isomorphism 
\[
\Theta_{\psi, \theta} \colon 
\Psi_{ai}^{-\infty}\left(\overline{s}^{-1}(U) \cap \mathcal{V}, \overline{r}^{\ast} \Hom(E)\right)
\iso
\Psi_{ai}^{-\infty}\left(U \times W, \End(E_x)\right),
\]
defined by pulling back via $\psi^{-1}$, transferring $V$- to $W$-coordinates with $\alpha$,
and applying $\vartheta$ to identify coefficients with $\End(E_x)$, see also \cite{NWX}.

\begin{Prop}
Let $Q\in \Psi^{m}_{ai,c}(\G \dtimes \Rr, \overline{r}^{\ast} E)$ and let
$k_Q$ be its compactly supported Schwartz kernel (with values in
$\Hom(\overline{r}^{\ast} E)$). Then, shrinking $\Omega$ if necessary, there exists
a symbol
\[
q \in S^{m}_{ai}\left(U \times \Rr^{d+1}; \End(E_x)\right),
\ d := \mathrm{rk}\,\A_x,
\]
such that, in the local coordinates $(u,\zeta,t)\in U \times W$ (with $\zeta \in \A_x$),
\[
\Theta_{\psi,\theta}(k_Q)\ =\ k_q,
\]
where
\[
k_q(u;\zeta,t) = (2\pi)^{-(d+1)} \int_{\Rr^{d}} \int_{\Rr} e^{-i(\scal{\zeta}{\xi} + t \tau)} q(u,\xi,\tau)\, d\xi\, d\tau \in \End(E_x).
\]
Equivalently,
\[
k_{Q|_{\Omega \cap \mathcal{V}}} = \Theta_{\psi,\theta}^{-1}(k_q).
\]
The symbol $q$ is unique modulo $S^{-\infty}_{ai}$ and depends continuously on $Q$.
If $Q$ is polyhomogeneous, then $q$ is classical/polyhomogeneous in the anisotropic
sense.
\label{Prop:localizer}
\end{Prop}

Write
\[
\Sigma^{m,2}\left(\A^\ast \oplus \Rr_M; \End(E)\right) := S^{m}_{ai} / S^{m-1}_{ai}.
\]
An element $[q] \in \Sigma^{m,2}$ is represented by a fiberwise parabolically homogeneous, $\tau$-holomorphic function $q_x(y, \xi, \tau)$ of degree $m$, defined for $(\xi, \tau) \neq (0,0)$ and unique modulo $S^{m-1}_{ai}$. The multiplication on $\Sigma^{\bullet, 2}$ is pointwise composition in $\End(E)$: $[q_1] \cdot [q_2] := [q_1 q_2]$, well-defined since $S^{m-1}_{ai}$ is an ideal.
For $Q \in \widehat{\Psi}^{m}_{ai}(\G \dtimes \Rr; \overline r^\ast E)$, choose $q \in S^{m}_{ai,ph}$ with $Q = \Op_{ai}(q) \bmod \Psi^{-\infty}_{\G\dtimes\Rr, ai}$ and set
\[
\sigma_m(Q) := [q_m]\ \in \Sigma^{m,2},
\]
where $q_m$ denotes any leading homogeneous component of $q$; this is chart-independent and well-defined modulo $S^{m-1}_{ai}$. Moreover, $\sigma_{m+m'}(Q_1 Q_2)=\sigma_m(Q_1) \cdot \sigma_{m'}(Q_2)$.

\begin{Def}
A Volterra operator $Q \in \widehat{\Psi}_{ai}^m(\G \dtimes \Rr; \overline{r}^{\ast} E), \ m \in \Zz$ is a continuous operator 
\[
Q \colon C_c^{\infty, 0}(\overline{r}^{\ast} E) \to C^{\infty, 0}(\overline{r}^{\ast} E)
\]
such that
\begin{enumerate}
\item $Q$ has the Volterra property.
\item $Q = \Op_{ai}(q) + R$ for some symbol $q \in S_{ai}^m(\A^{\ast} \oplus \Rr_M, \overline{r}^{\ast} E)$ and some $R \in S_{\rho, +}(\G \dtimes \Rr, \overline{r}^{\ast} E)$.
\end{enumerate}
\label{Def:VolterraCalc}
\end{Def}

\begin{Rem}
Define the Payley-Wiener algebra as in \cite{LR} by $C_{PW}^{\infty}(\A^{\ast} \oplus \Rr_M) := \Ff C_c^{\infty}(\A \oplus \Rr_M)$. Given a cutoff $\chi \in C_c^{\infty}(\A \oplus \Rr_M)$ with $\chi \equiv 1$ near $(\zeta, t) = (0, 0)$, then $\widehat{\chi} := \Ff \chi \in C_{PW}^{\infty}$. We can rescale the kernel, preserving compact support $\chi_{\epsilon}(\zeta, t) := \epsilon^{-(d +2)} \chi(\epsilon^{-1} \zeta, \epsilon^{-2} t)$, $\chi_{\epsilon} \in C_{PW}^{\infty}$ and $q \mapsto \widehat{\chi}_{\epsilon} q$ furnishes a microlocal cutoff, leaving principal symbols unchanged. We set $C_{PW, +}^{\infty} := \Ff C_c^{\infty}(\A \oplus \Rr_M; t \geq 0)$ which is a Fr\'echet algebra with respect to pointwise composition, i.e. the Volterra property is preserved.
\label{Rem:PWalgebra}
\end{Rem}

\begin{Prop}
\emph{i)} The compactly supported calculus $\Psi_{ai,c}^{m}(\G \dtimes \Rr; \overline{r}^{\ast} E)$ and the extended calculus $\widehat{\Psi}_{ai}^m(\G \times \Rr; \overline{r}^{\ast} E)$ are closed with respect to composition that is given by groupoid convolution of reduced Schwartz kernels.

\emph{ii)} Let $P \in \widehat{\Psi}_{ai}^m(\G \dtimes \Rr; \overline{r}^{\ast} E)$ be such that its Volterra principal symbol is pointwise invertible (i.e. $\sigma_m(P)$ invertible in the symbol space $\Sigma^{m,2}$). Then there is $Q \in \widehat{\Psi}_{ai}^{-m}(\G \dtimes \Rr; \overline{r}^{\ast} E)$ such that $P Q = \Id - S_1, \ Q P = \Id - S_2$ with $S_1, S_2 \in S_{\rho, +}(\overline{r}^{\ast} E)$. 
\label{Prop:VolterraCalc}
\end{Prop}

\begin{proof}
\emph{i)} Working in a chart provided by the localizer $\Theta_{\psi,\theta}$ by virtue of Prop. \ref{Prop:localizer}, let $Q_j=\Op_{ai}(q_j)$ with $q_j\in S^{m_j}_{ai,ph}(U \times \Rr^{d+1};\End(E_x))$ ($j=1,2$). Their reduced kernels are
\[
k_{q_j}(u; \zeta, t) = (2\pi)^{-(d+1)} \int e^{-i \scal{\zeta}{\xi} + t \tau)} q_j(u, \xi, \tau) \,d\xi\,d\tau,
\]
which vanish for $t<0$ by holomorphy in $\tau \in \Cc_-$. The composition kernel on $U$ is the
convolution $k_{q_1} \ast k_{q_2} = k_{q_1 \# q_2}$ is still supported in $\{t \geq 0\}$.
A standard computation shows that $q_1 \# q_2\in S^{m_1+m_2}_{ai,ph}$ with
\[
q_1 \# q_2\ \sim \sum_{\alpha} \frac{1}{\alpha!} \left(\partial_\xi^{\alpha} q_1 \right) \left(D_{\xi}^\alpha q_2\right),
\]
so $\Theta_{\psi,\theta}(Q_1 Q_2)  =\Op_{ai}(q_1\# q_2)$ modulo $\Psi^{-\infty}_{ai}$,
and the leading term gives $\sigma_{m_1+m_2}(Q_1 Q_2)=\sigma_{m_1}(Q_1) \cdot \sigma_{m_2}(Q_2)$. Proposition \ref{Prop:WeightSmoothVolterra} handles every term containing a residual factor and preserves the causal sector. 
\emph{ii)} Let $p := \sigma_m(P) \in \Sigma^{m,2}$ be pointwise invertible; write $p^{-1}$ for its
fiberwise inverse in $\Sigma^{-m,2}$. In a localizer chart, choose a representative
$p_m(u, \xi, \tau)$ of $p$ and define the first approximation $q_{-m}^{(0)}(u, \xi, \tau) := p_m(u,\xi, \tau)^{-1}\in S^{-m}_{ai}$. Define inductively $q_{-m-j}^{(0)} \in S^{-m-j}_{ai}$ so that, for each $N$,
\[
p\# \left(\sum_{j < N} q^{(0)}_{-m-j} \right) = 1 \ \bmod \ S^{-m-N}_{ai},
\]
with the anisotropic $\#$-product from \emph{(i)}.
This yields a formal inverse $q^{(0)} \sim \sum_{j \geq 0} q^{(0)}_{-m-j} \in S^{-m}_{ai,ph}$,
but the cutoffs used to globalize in $(\xi, \tau)$ need not preserve $\tau$-holomorphy,
so $q^{(0)}$ may fail to be Volterra. To make the kernel compactly supported and keep causality, pick $\chi \in C_c^\infty(\A \oplus \Rr_M)$ with $\chi \equiv 1$ near $(\zeta, t) = (0, 0)$ and $\supp \chi \subset \{t \geq 0\}$, and set $\widehat{\chi} := \F_{fib} \chi \in C_{PW,+}^\infty$. Replace $q$ by $\widehat{\chi} q$ and, for each homogeneous piece $q^{(0)}_{-m-j}$, apply the fiberwise Volterra extension given by Lem. \ref{Lem:Greiner} to get a unique homogeneous distribution $g_{-m-j}$ of degree $-m-j$ whose $(\xi, \tau)$-inverse Fourier transform is supported in $\{t \geq 0\}$, and which agrees with $q^{(0)}_{-m-j}$ off $(\xi, \tau) = (0,0)$. By a standard Borel summation as in \cite[p. 5-6]{PongeMika}, there exists
\[
q \in S^{-m}_{ai,ph} \ \text{such that} \
q \sim \sum_{j \geq 0} g_{-m-j}
\]
and $q$ is holomorphic in $\tau \in \Cc_-$; hence $\Op_{ai}(q)$ is Volterra.
 Define $Q:=\Op_{ai}(\widehat{\chi} q)$. By construction, $P Q = \Id - R_1$ and $Q P = \Id - R_2$ with compactly supported smooth Volterra remainders. Hence $R_1, R_2 \in S_{\rho, +}$ by Proposition \ref{Prop:WeightSmoothVolterra}.
\end{proof}

\begin{Lem}
Let $R \in S_{\rho, +}(\G \dtimes \Rr)$. On every bounded time interval the Volterra series
\[
(\mathrm{Id} - R)^{-1} = \mathrm{Id} + \sum_{j \geq 1} R^{\ast j}, 
\]

where $R^{\ast j}$ denotes the $j$-fold convolution power. The series converges in the Fr\'echet topology induced by $S_{\rho}$ on that interval. Its remainder belongs to $S_{\rho, +}(\G \dtimes \Rr)$. 
\label{Lem:CausalResidualInverse}
\end{Lem}

\begin{proof}
Fix $T > 0$ and restrict kernels to $0 \leq t \leq T$. Write the convolution on $\G \dtimes \Rr$ as
\[
(k_1 \ast k_2)(\gamma, t) = \int_{\G_{s(\gamma)}} \int_0^t k_1(\gamma \eta^{-1}, t-s) k_2(\eta, s)\,ds\,d\mu_{s(\gamma)}(\eta)
\]

for causal kernels. Iteration yields an expression for the time component of $R^{\ast j}$ as an integral over
\[
\Sigma_{j-1}(t) = \left\{(s_1, \ldots, s_{j-1}): 0 \leq s_1 \leq \cdots \leq s_{j-1} \leq t\right\}.
\]

The Euclidean volume of that equals
\[
\frac{t^{j-1}}{(j-1)!}.
\]

Let $p$ be one of the graph seminorms defining the topology of $S_{\rho}$. Continuity of convolution in the ambient smoothing algebra, together with the Leibniz formula
\[
\delta_{\rho}^{\alpha} (k_1 \ast k_2) = \sum_{\beta \leq \alpha} \binom{\alpha}{\beta} (\delta_{\rho}^{\beta} k_1) \ast (\delta_{\rho}^{\alpha - \beta} k_2),
\]

shows that there are a graph seminorm $q$, an integer $m_p$, and a constant $C_p > 0$ such that
\[
p(R^{\ast j}) \leq C_p^j j^{m_p} q(R)^j \frac{T^{j-1}}{(j-1)!}.
\]

Since $R$ is fixed, enlarge the constant and write
\[
p(R^{\ast j}) \leq C_{p, R}^j j^{m_p} \frac{T^{j-1}}{(j-1)!}.
\]
The resulting series converges. The same estimate applies to the seminorm in the defining countable family. Hence 
\[
\sum_{j \geq 1} R^{\ast j}
\]
converges in the induced Fr\'echet topology on $[0,T]$. The limit has causal support, because its components have causal support. By completeness of $S_{\rho}$, the previous series is contained in $S_{\rho}$. For the partial sums
\[
S_T^{(N)} = \Id + \sum_{j=1}^{N} R^{\ast j}
\]

one has 
\[
(\Id - R) \ast S_T^{(N)} = \Id - R^{\ast (N+1)}
\]

and likewise for the left multiplier. The preceding estimates imply $R^{\ast (N+1)} \longrightarrow 0$ in every graph seminorm. Passing to the limit yields the assertion.
\end{proof}

\begin{Lem}
Let $R \in S_{\rho, +}(\G \dtimes \Rr)$ and write its reduced kernel in standard coordinates as $k_R(u; \eta, t)$. For every compact coordinate set, every longitudinal differential operator $D$, and every $N \in \Nn$, one has
\[
D \left[\hbar^{-d-2} k_R(u; \hbar \zeta, \hbar^2 t)\right] = O(\hbar^N)
\]

uniformly in the transverse parameter as $\hbar \downarrow 0$. The same estimate holds after composing $R$ on either side with a fixed operator in $\widehat{\Psi}_{ai}^m$. 
\label{Lem:ParabolicFlatResidual}
\end{Lem}

\begin{proof}
The time variable belongs to the longitudinal variables of $\G \dtimes \Rr$. By \cite[Thm. 10]{LV}, the kernel is smooth in that variable and depends continuously on the transverse parameter. Causality gives $k_R(u; \zeta, t) = 0$ for $t < 0$. Smoothness in $t$ then implies 
\[
\partial_t^j k_R(u; \eta, 0) = 0
\]

for every $j$. Fix a compact set $K \subset \{(u, \zeta, t) : t \geq 0\}$. Let $D$ have longitudinal order $r$ on $\G \dtimes \Rr$. Applied to $\hbar^{d-2} k_R(u; \hbar \zeta, \hbar^2 t)$ each $\zeta$-derivative contributes a factor $\hbar$ and each $t$-derivative contributes $\hbar^2$. Thereby, it is sufficient to estimate expressions of the form 
\[
\hbar^{-d-2+a+2b} (\partial_{\zeta}^{\beta} \partial_t^b k_R)(u; \hbar \zeta, \hbar^2 t),
\]

where $a = |\beta|$. Choose $M \in \Nn$ sufficiently large with $2M - d - 2 + a + 2b \geq N$. All time derivatives of $\partial_\zeta^\beta \partial_t^b k_R$ vanish at $t = 0$. By Taylor's remainder formula
\[
(\partial_\eta^\beta \partial_t^b k_R)(u; \hbar \zeta, \hbar^2 t) = \frac{\hbar^2t}{(M-1)!} \int_0^t (1 - s)^{M-1} (\partial_\zeta^\beta \partial_t^{b+M} k_R)(u; \hbar \zeta, \hbar^2 t)\, ds.
\]

The last derivative is bounded uniformly for $(u, \zeta, t) \in K$ and $0 < \hbar \leq 1$. With the given choice of $M$, this yields the required $O(\hbar^N)$ estimate. 

Next, let $P \in \widehat{\Psi}_{ai}^m$, then by Proposition \ref{Prop:WeightSmoothVolterra} we have that $P \ast R$ and $R \ast P$ are each contained in $S_{\rho,+}$. Applying the argument above to these convolutions furnishes the other assertion.
\end{proof}

\section{Short time asymptotic expansion}


\subsection{Parabolic adiabatic deformation groupoid}

In order for us to specify the parabolic adiabatic deformation groupoid, we recall some of the necessary notions that enter into the filtered calculus as described in \cite{vEY}; see also \cite{DH, SLLS}. 
We fix the filtration $H^{\bullet}$ of maximal degree $N = 2$ on the Lie algebroid $\A \oplus \Rr_M \cong \A(\G \dtimes \Rr)$ of the groupoid $\G \dtimes \Rr$ that is given by $H^0 = \{0\}, \ H^1 = \A \oplus 0, \ H^2 = \A \oplus \Rr$. Fix the notation
\[
t^p H := H^p / H^{p-1}, \ \A(\G \dtimes \Rr)_H := \oplus_{i=0}^N H^i / H^{i-1} = \oplus_{i=0}^N t^i H. 
\]

Then $\A_{H, x} = \oplus_p t_x^p H$ has a canonical Lie algebra structure depending smoothly on the base point $x$. Let $\nabla$ denote a linear connection on $\A \oplus \Rr_M$ preserving the grading automorphism 
\[
\delta_{\lambda} \colon \A(\G \dtimes \Rr)_H \to \A(\G \dtimes \Rr)_H, \ \delta_{\lambda} = 
\oplus_{i=1}^{N} \lambda^{w_i}, \ w_1 = 1, \ w_2 = 2. 
\]
The $\delta_{\lambda,x}$ result from integration of the corresponding scaling automorphisms $\dot{\delta}_{\lambda, x}$ via the fiberwise exponential map 
\[
\exp^{\nabla} \colon \A(\G \dtimes \Rr)_H \to (\G \dtimes \Rr)_{H}; 
\]
the right hand side is a bundle of Lie groups and $\exp^{\nabla}$ implements a diffeomorphism of locally trivial bundles over $M$. Thereby, $\dot{\delta}_{\lambda, x} \in \Aut(t_x H)$ integrates to $\delta_{\lambda, x} \in \Aut((\G \dtimes \Rr)_{H, x})$. In particular 
\[
\delta_{\lambda} \colon (\G \dtimes \Rr) \to (\G \dtimes \Rr) 
\]
are bundle diffeomorphisms such that 
\[
\delta_{\lambda} \circ \exp = \exp \circ \dot{\delta}_{\lambda}. 
\]
Fix the splitting $\Psi \colon \A(\G \dtimes \Rr)_H \to \A \oplus \Rr_M$ corresponding to $\nabla$. This allows us to identify $\A(\G \dtimes \Rr)_H \cong \A \oplus \Rr_M$ via $-\Psi$ and with respect to the left-invariant vector fields on the bundles of Lie groups. Also, $(\G \dtimes \Rr)_{H, x} \cong \A_{H, x}$ via $\exp$ and with respect to the right-invariant vector fields on the Lie algebroid. We want to define the corresponding deformation Lie groupoid that is given as a set by the expression \eqref{DeformationGroupoid}. This groupoid is obtained by integration of the appropriate Lie algebroid of the deformation groupoid that is given by the set and the Lie algebra of vector fields:
\begin{align*}
& \U_{H^{\bullet}} := (\A \oplus \Rr_M) \times \Rr_{\hbar}^{\ast} \cup \A(\G \dtimes \Rr)_H \times \{0\}_{\hbar}, \\
& \Gamma(\U_{H^{\bullet}}) := \{X \in \Gamma((\A \oplus \Rr_M) \times \Rr_{\hbar}) : \partial_{\hbar}^i X \in \Gamma(H^i), \ i \geq 0\}. 
\end{align*}

We have the diagram
\[
\begin{tikzcd}[row sep=huge, column sep=huge, text height=2ex, text depth=0.5ex]
(\G \dtimes \Rr)_H \arrow{d} & \arrow[l, "\exp^{\nabla}"] \A(\G \dtimes \Rr)_H \arrow{d} \arrow[r, "-\Psi", "\simeq"'] & \A \oplus \Rr_M \arrow{d} \arrow{r} & M \times M \arrow[d, "\mathrm{pr}_1"] \\
M \ar[equal]{r} & M \ar[equal]{r} & M \ar[equal]{r} & M \arrow[bend right=60, swap]{u}{\Delta}
\end{tikzcd}
\]

Set $\nabla^{\Psi} := \Psi \circ \nabla \circ \Psi^{-1}$ for the induced connection on the filtered Lie algebroid $\A \oplus \Rr_M$ and 
\[
\Exp^{\nabla^{\Psi}} \colon \A \oplus \Rr_M \to \G \dtimes \Rr
\] 
the generalized exponential map. We have $\Gamma((\G \dtimes \Rr)^{pa}) \cong \Gamma(\A(\G \dtimes \Rr)_H \times \Rr_{\hbar})$, via $\widetilde{X} \mapsto ((x, \hbar) \mapsto \delta_{\hbar} \Psi^{-1} \widetilde{X}(x, \hbar))$ and the identification Lie algebra isomorphism 
\[
\Gamma(\A((\G \dtimes \Rr)^{pa})) = \{\widetilde{X} \in \Gamma(\A(\G \dtimes \Rr) \times \Rr_{\hbar}) : \partial_{\hbar|\hbar = 0}^k \widetilde{X} \in \Gamma(t^k H), \ k \in \Nn_0\}.
\]

The left hand side bracket is the bracked of $\A(\G \dtimes \Rr) \times \Rr_{\hbar}$ and the anchor is given by the anchor of $\A(\G \dtimes \Rr) \times \Rr_{\hbar}$. As in the definition of the standard (non-filtered) tangent groupoid, we indicate the smooth structure by charts of the form 
\[
(x, \xi, \tau, \hbar) \mapsto \begin{cases} \Exp_x^{\nabla^{\Psi}}(\Psi(\delta_{\hbar}(\xi, \tau), \hbar), \ \hbar \not= 0 \\
(x, \xi, \tau, 0), \ \hbar = 0 \end{cases}
\]

in a small tubular neighborhood of the zero section $O_M \subset \U \subset \A \oplus \Rr_M$. Denote by $(\G \dtimes \Rr)_H^{op}$ the groupoid with the opposite categorical structure. We have the evaluations
\[
\Gamma(\A((\G \dtimes \Rr)_H^{op})) \xleftarrow{\mathrm{ev}_0} \Gamma(\A((\G \dtimes \Rr)^{pa})) \xrightarrow{\mathrm{ev}_t} \A(\G \dtimes \Rr),
\]

and $\mathrm{ev}_0(\widetilde{X}) = \lim_{\hbar \to 0} \delta_{\hbar} \mathrm{ev}_{\hbar}(\widetilde{X})$. We have $\Gamma(\A(\G \dtimes \Rr)_H) = \Gamma(\A((\G \dtimes \Rr)_H^{op}))$ with respect to the left invariant vector fields of the Lie algebra on the right hand side, and the inclusions $(\G \dtimes \Rr)_{H, x} \hookrightarrow (\G \dtimes \Rr)_H$ are strict groupoid morphisms, $x \in M$. 

\subsection{Standard coordinates and lift}

We fix the notion of standard exponential coordinates on $(\G \dtimes \Rr)^{pa}$ that are obtained via a localizer associated to the generalized exponential map as described in the previous section. These are given by the $4$-tuple $(U, V, \varrho^{pa}, \U)$ where $U \times V \times \Rr_{\hbar} \supset \U \supset U \times \{0\} \times \{0\}_{\hbar}$ and $\U$ is invariant under the scaling action
\[
\alpha_{\lambda}(u, z, \hbar) = (u, \delta_{\lambda} z, \lambda^{-1} \hbar), \ (u, z, \hbar) \in \U, \ \lambda > 0. 
\]

The map $\varrho$ is the specialization of the Lie algebroid anchor of $\A \oplus \Rr_M$ with respect to the algebroid parametrization; see also \cite{McDonald}. 


Let $\{\mu_x\}_{x \in M}$ be a Haar system for $\G$, and endow $\G \dtimes \Rr$ with the product Haar system $\{\mu_x \otimes dt\}_{x \in M}$. On $(\G \dtimes \Rr)^{pa}$ we record the following specialization:
\begin{itemize}
\item For $\hbar \neq 0$, in the chart $U \times V$ with coordinates $(x, z)$, we pull back the Haar density by
\begin{align*}
& \Exp_x^{\nabla^{\Psi}} \circ \delta_{\hbar} \colon V \to (\G \dtimes \Rr)_{x}, \\
& \delta_{\hbar}(z) := (\hbar \zeta, \hbar^2 t), \ \text{if} \ z = (\zeta, t),
\end{align*}

and write
\begin{align*}
& d\mu_{x, \hbar}(z) = J_x(z, \hbar)\,dz, \
J_x(z, \hbar) := \left|\det D \left(\Exp_x^{\nabla{\Psi}} \circ \delta_{\hbar}\right) \right|.
\end{align*}

\item For $\hbar = 0$, on the osculating group $(\G \dtimes \Rr)_{H, x} \cong \A_x \oplus \Rr$ we take the left Haar measure given by the product of a fixed smooth density on $\A_x$ and Lebegue measure $dt$, so in the $(\zeta, t)$-coordinates
\[
d\mu_{x, 0}(\zeta, t) = d\zeta \, dt.
\]
\end{itemize}

These choices are compatible with parabolic scaling:
\[
(\delta_{\lambda})^{\ast}(d \mu_{x, \hbar}) = \lambda^{d+2} \,d\mu_{x, \lambda \hbar}, \ \lambda > 0
\]

and glue to a smooth Haar system $\{\mu_{x, \hbar}\}_{(x, \hbar) \in M \times \Rr_{\hbar}}$ on $(\G \dtimes \Rr)^{pa}$. 

Let $q \in S_{ai, ph}^m(\A^{\ast} \oplus \Rr_M; \End(E))$ be a Volterra symbol for $T$ in the local chart. Define the \emph{$\hbar$-rescaled symbol}
\begin{align}
q_{\hbar}(u, \xi, \tau) := q\left(u, \hbar \xi, \hbar^2 \tau\right), \ \hbar \geq 0, \label{paSymbol}
\end{align}

where for $\hbar = 0$ the right-hand side is understood as the (homogeneous) principal limit on $(\xi, \tau) \not= (0, 0)$ (i.e. the Volterra principal symbol in degree $m$ on the osculating fibers). Then the quantization on $(\G \dtimes \Rr)^{pa}$ is given, in standard coordinates, by the fiberwise oscillatory integral
\begin{align}
k_q^{pa}(u, \zeta, t, \hbar) &:= (2 \pi)^{-(d+1)} \int_{\Rr^d} \int_{\Rr} e^{-i \scal{\zeta}{\xi} + i t \tau} q_{\hbar}(u, \xi, \tau) \,d\xi\,d\tau \label{paQuant}
\end{align}

which defines a family of reduced kernels supported in $\{t \geq 0\}$. Equivalently, on the kernel side this is just the parabolic rescaling
\begin{align}
& k_q^{pa}(u, \zeta, t, \hbar) = \hbar^{-d-2} k_q(u, \zeta \hbar, t \hbar^2), \ \hbar > 0 \notag \\
& k_q^{pa}(u, \zeta, t, 0) = k_{q_m}(u, \zeta, t), \label{paRescale}
\end{align}

where $k_q$ is the reduced kernel of $T = \Op_{ai}(q)$ and $q_m$ is the leading homogeneous component of $q$. 

We continue to work in the standard coordinates, i.e. with respect to a groupoid parametrization of $(\G \dtimes \Rr)^{pa}$ that is refined via the generalized exponential map $\Exp^{\nabla^{\Psi}}$ with respect to a splitting $\Psi$. 
In these standard coordinates, we study compactly supported continuous families $f \in C(U \times \Rr_{\hbar}, C_c^{\infty}(V))$. The goal is to extend a given Volterra operator $T$ to a $1$-parameter family on the parabolic adiabatic deformation grouopoid $(\G \dtimes \Rr)^{pa}$. The family $\{T_{\hbar}\}_{\hbar > 0}$ acts, for $s$-fibered (longitudinally smooth, transversely continuous) distributional kernels $k \in \mathscr{E}_{s}^{', 0}$ via
\begin{align}
(T_{\hbar} g)(u) &= \int_V k_q^{pa}(u, z, \hbar) g(\exp(\varrho^{pa} \delta_{\hbar} z) u) \,dz \notag \\
&= \hbar^{-d-2} \int_V k_q^{pa}(u, \delta_{\hbar}^{-1} z, \hbar) g(\exp(\varrho^{pa} z) u)\,dz, \label{paVolterra}
\end{align}

and for $\hbar = 0$ it is the model operator on the osculating bundle
\[
(T_0 g)(u) = \int_{\A_x \times \Rr} k_{q}^{pa}(u, \zeta, t, 0) g\left(\exp(\varrho^{pa}(\zeta, t))u \right)\, d\zeta\,dt.
\]
Recall that by an application of the scaling action (cf. \cite{DebordSkandalis}) where we write $k = k^{pa}$
\[
((\alpha_{\lambda})_{\ast} k, \varphi)(u, \hbar) = \lambda^{d+2}(k, \varphi \circ \alpha_{\lambda})(u, \lambda \hbar), \ (u, \hbar) \in r(\U),
\]

where the $s$-fiberwise action of $k$ is given by
\[
(k, g)(u, \hbar) = \int_V k(u, z, \hbar) g(u, z, \hbar) \,dz, \ g \in C^{\infty, 0}(\U). 
\]

The extended Volterra operator family kernel $k$ then fulfills the homogeneity property
\begin{align}
& \forall_{\lambda > 0} \ \lambda^{-m-d-2} (\alpha_{\lambda})_{\ast} k - k \in C^{\infty, 0}(\U). \label{Homogeneity}
\end{align}

We make use of these expressions in standard coordinates to study the asymptotic expansion of the heat kernel on a Lie groupoid. Assume that the chosen integration is compatible with the boundary defining functions in the sense of Definition \ref{Def:WeightSmoothVolterra}.

\begin{Prop}
Let $A \in \Diff^m(\G; r^{\ast} E)$ be a $\G$-differential operator acting on sections of a vector bundle $\pi \colon E \to M$ on the $d$-dimensional compact manifold (possibly with corners) $M$ of order $m > 0$ which is self-adjoint and elliptic, acting as an unbounded operator on the $C_r^{\ast}(\G)$-module $\E := C_r^{\ast}(\G, r^{\ast} E)$ and non-negative with respect to the canonical $L^2 = H^0$ inner product. Then the semigroup heat kernel $k_t$ of $\exp(-tA)$ has a short time asymptotic expansion when restricted to the diagonal $\Delta_M$ in $\G$ of the form
\[
k_{t|\Delta} \sim \sum_{j=0}^{\infty} t^{\frac{j-d}{m}} q_j, \ \text{as} \ t \to 0^{+}. 
\]

Here the coefficients are sections $q_j\in \Gamma\left(\End(E)\otimes |\!\wedge|\A^\ast\right)$ that depend continuously on $x \in M$. 
\label{Prop:AsymptExp}
\end{Prop}

\begin{proof}
Denote by $C_+(\Rr, H^0)$ the set of continuous functions $\psi \colon \Rr \to H^0(r^{\ast} E)$ for which $\psi(t) = 0$ on a left half-line, cf. \cite{Greiner}. Define the causal fundamental solution by
\[
(Q \psi)(s) := \int_{-\infty}^s e^{-(s -t) A} \psi(t)\,dt.
\]

Proposition \ref{Prop:HeatSemigroup} gives $Q \psi \in C_+(\Rr, H^0)$. If $\supp \ \psi \subset [t_0, \infty)$, contractivity yields
\[
\|(Q \psi)(s)\| \leq (s - t_0)_{+} \sup_{t \in [t_0, s]} \|\psi(t)\|. 
\]

For $t > 0$, Proposition \ref{Prop:HeatSemigroup} places $e^{-tA}$ in the smoothing algebra. Its reduced kernel is longitudinally smooth and transversely continuous by \cite[Thm. 10]{LV}. The causal kernel of $Q$ is therefore a $C^{\infty, 0}$ kernel away from $t = 0$ and defines an $s$-fibered distribution across $t = 0$. 

The operator $A + \partial_t$ is Volterra elliptic. Proposition \ref{Prop:VolterraCalc} gives a causal parametrix $\widetilde{Q} \in \widehat{\Psi}_{ai, ph}^{-m}$ with
\[
(A + \partial_t) \widetilde{Q} = \mathrm{Id} - S_1, \qquad S_1 \in S_{\rho, +}.
\]
 
On a bounded time interval, Lemma \ref{Lem:CausalResidualInverse} yields
\[
(\mathrm{Id} - S_1)^{-1} = \mathrm{Id} + R_1, \qquad R_1 \in S_{\rho, +}.
\]

The operator $\widetilde{Q} (\mathrm{Id} + R_1)$ is the causal right inverse of $A + \partial_t$ and agrees with $Q$ by uniqueness of the causal fundamental solution. It belongs to the extended Volterra calculus. We write $\widetilde{Q}_{\hbar}$ and $Q_{\hbar}$ for the corresponding lifts to the parabolic adiabatic groupoid. 
 
Denote by $\mathcal{P}^s$ the space of families of polynomial volume densities on the fibers of $(\G \dtimes \Rr)_{H, x} \cong \A(\G \dtimes \Rr)_{H, x}$, which are smooth in the fiber variables and continuous in $x$ with respect to all fiberwise $C^{\infty}$ seminorms, and which have homogenous degree $s$. Thus 
\[
\mathcal{P}^s((\G \dtimes \Rr)_H) := \{k \in \Psi_{ai}^{-\infty}((\G \dtimes \Rr)_H; \widetilde{r}^{\ast} E)|(\delta_{\lambda})_{\ast} k = \lambda^s k, \ \forall \ \lambda > 0\}. 
\]

Express the lifted family $\widetilde{Q}_{\hbar}$ in fixed standard coordinates $\U$. Then the task is to study the asymptotic behavior of the localizer $\Theta^{\ast}(k_{\widetilde{Q}_{\hbar}|\U})$ at $\hbar = 0$ as $t \to 0^{+}$. 
%
%
We have
\[
\Theta^{\ast}(k_{\widetilde{Q}_{\hbar}|\U}) - \sum_{j=0}^{J_N} q_j \in C^{N, 0}
\]

and by \cite[Lem. 2]{DH}
\begin{align}
& (\delta_{\lambda}^{(\G \dtimes \Rr)_H})_{\ast} q_j = \lambda^{-m-d-2-j} q_j - \lambda^{-m-d-2-j} \log |\lambda| p_j, \label{Homogeneity}
\end{align}

where $p_j$ is $C^{\infty,0}$ and strictly homogeneous, i.e. 
\[
p_j \in \mathcal{P}^{-m-j-d}((\G \dtimes \Rr)_H; \End(E))
\]

such that
\[
(\delta_{\lambda}^{(\G \dtimes \Rr)_H)})_{\ast} p_j = \lambda^{-m-d-j} p_j, \ \lambda \not= 0.
\]

Here the representatives $q_j$ satisfying \eqref{Homogeneity} are unique modolu $\mathcal{P}^{-m-d-j}((\G \dtimes \Rr)_H; \End(E))$. The correction satisfies
\[
Q - \widetilde{Q} = \widetilde{Q} R.
\]

Lemma \ref{Lem:ParabolicFlatResidual} yields $O(\hbar^N)$ for this term in every $C^{k,0}$ seminorm and for arbitrary $N$. The exact fundamental solution and the Volterra parametrix therefore have the same parabolic asymptotic coefficients
\[
\Theta^{\ast}(k_{Q_{\hbar}|\U}) \sim \sum_{j=0}^{\infty} q_j.
\]

Taylor's theorem applies in the longitudinal variables with continuous dependence on the transverse parameter. It permits the usual change of representatives by elements of $\mathcal{P}^{-m-d-j}$. For every $N$ there is $j_N$ such that
\[
\left(\Theta^{\ast}(k_{Q_{\hbar = 0|\U}}) - \sum_{j=0}^{j_N} q_j\right)(\widetilde{\zeta}) = O(|\widetilde{\zeta}|^N)
\]

as $\widetilde{\zeta} \to \widetilde{o}_x \in (\G \dtimes \Rr)_{H, x}$, uniformly for $x$ in compact subsets of $M$. 
The coefficients can be chosen as the causal homogeneous extensions of the corresponding homogeneous Volterra symbol components. Lemma \ref{Lem:Greiner} yields a unique extensions with this support property and preserves strict parabolic homogeneity. Comparing that strict homogeneity with \eqref{Homogeneity} furnishes
\[
\log |\lambda| p_j = 0
\]

for every $\lambda > 0$. Choosing $\lambda \not= 1$ yields $p_j = 0$. The asymptotic expansion contains strictly homogeneous coefficients. Expressing everything in the localizer with respect to standard coordinates, we find:
\[
k_Q(x, s -t) = \sum_{j=0}^{j_N} \Theta(\check{q}_{j,x}(o_x, t -s)) + O(|t|^{\frac{N}{r}})
\]

as $|t| \to 0$ uniformly in $x$ and, making use of the homogeneous dimension $m + d + 2$ of $M \times \Rr$
\[
\Theta_x(\check{q}_{j,x}(o_x, t)) = t^{\frac{j-d-2}{m}} \Theta_x(\check{q}_x(o_x, 1))
\]

for each $x, t > 0$ where we made use of the left-trivialization of $(G \dtimes \Rr)_{H, x}$ to identify the two sides. Setting $\check{q}_j(x) \,dt := \Theta_x(\check{q}_{j,x}(o_x, 1))$ furnishes the result.
\end{proof}

\end{document}